\documentclass[10pt]{article}
\usepackage{amsmath,amsfonts,latexsym, amssymb}
\usepackage{mathrsfs}
\usepackage{pstricks}
\usepackage{pst-plot}
\usepackage{pst-node}

\def\endproofbox{\hskip 1.3em\hfill\rule{6pt}{6pt}}
\newenvironment{proof}%
{%
\noindent{\it Proof.}
}%
{%
 \quad\hfill\endproofbox\vspace*{2ex}
}
 \newcommand{\rmv}[1]{}

 \def \cH {{\cal H}}

 \def\a{\alpha}

 \def\S{\Sigma}

 \def\sI{\mathscr{I}}

 \def\sM{\mathscr{M}}

\newtheorem{thm}{Theorem}[section]

\newtheorem{lemma}[thm]{Lemma}
\newtheorem{cor}[thm]{Corollary}

\def\endproofbox{\hskip 1.3em\hfill\rule{6pt}{6pt}}

   \rmv{

\newenvironment{proof}%
{%
\noindent{\it Proof.}
}%
{%
 \quad\hfill\endproofbox\vspace*{2ex}
}

    }

\def\to{\rightarrow}
\def\toto{\longrightarrow}

\def\IFFF{\Longleftrightarrow}

 \def\wt{\mbox{wt}}
 \def\imp{\mbox{imp}}
 \def\BAL{\textsf{BAL}}

\def\R{0.14} 
\def\Thickness{1.8pt} 
\def\Hzero{0.0} 
\def\Hone{0.2} 
\def\Htwo{0.4}
\def\Hthree{0.6}
\def\Hfour{0.8}
\def\LINETYPE{*-*}
\def\PicWidth{5} 
\def\PicHeight{1.4} 
\def\LabelSpacing{-0.4} 
\def\ClusterColor{lightgray}
\def\SideColor{gray}
\def\RootColor{black}

\newgray{gray1}{.1}
\newgray{gray2}{.2}
\newgray{gray3}{.3}
\newgray{gray4}{.4}
\newgray{gray5}{.5}
\newgray{gray6}{.6}
\newgray{gray7}{.7}
\newgray{gray8}{.8}
\newgray{gray9}{.9}

\def\Wbtwo{.5}
\def\Wbone{1}
\def\Wmid{1.5}
\def\Wpone{2}
\def\Wptwo{2.5}
\def\PicWidthL{3.7}
\def\ThicknessTwo{1.2pt}

\def\zero{0}
\def\one{0.3}
\def\two{.6}
\def\three{.9}
\def\four{1.2}
\def\five{1.5}
\def\six{1.8}
\def\seven{2.1}
\def\eight{2.4}
\def\nine{2.7}
\def\ten{3.0}
\def\eleven{3.3}
\def\twelve{3.6}
\def\thirteen{3.9}
\def\fourteen{4.2}
\def\fifteen{4.5}

 \rmv{

 \headsep0.3in     

 \pagestyle{myheadings}

 \markboth{Robert E. Jamison \hspace{.4 in}
 >>> Latex ignores this <<<  \hspace{.4 in}
 >>> Latex ignores this <<<  Version:  $8^{th}$ MAY, 2008}
  {Beyerl and Jamison  \hspace{.1 in} Balanced Interval Graphs \hspace{.1 in}}

 }

 \author{Jeffery J. Beyerl
 \thanks{Department of Mathematics, Clemson University, Clemson, SC 29634-0975
 \mbox{ email: \textsf{jbeyerl@clemson.edu}}}
 \and
 Robert E. Jamison
 \thanks{Department of Mathematics, Clemson University, Clemson, SC 29634-0975
 \mbox{ email: \textsf{rejam@clemson.edu}}}
 \thanks{Affiliated Professor, University of Haifa}
 }

 \title{Interval Graphs with \\ Containment Restrictions}

\begin{document}

     \date{}

 \maketitle


 \begin{abstract}
 An interval graph is {\em proper} iff it has a representation in which
 no interval contains another.  Fred Roberts \cite{R} characterized the proper interval graphs
 as those containing no induced star $K_{1,3}$.  Proskurowski and
 Telle \cite{PT} have studied $q$-proper graphs, which are interval graphs
 having a representation in which no interval is properly contained
 in more than $q$ other intervals. Like Roberts they found that their classes of graphs
 where characterized, each by a single minimal forbidden subgraph. This paper initiates the study of
 $p$-{\em improper interval graphs} where no interval contains
 more than $p$ other intervals.  This paper will focus on a special
 case of $p$-improper interval graphs for which the minimal
 forbidden subgraphs are readily described.  Even in this case, it
 is apparent that a very wide variety of minimal
 forbidden subgraphs are possible.
 \end{abstract}


    \section{Introduction}  \label{intro}

 A finite, simple graph $G=(V,E)$ is an {\em interval graph} iff there is
 an assignment $\a: v \toto I_v$ of vertices $v$ of $G$ to intervals $I_v$
 on the real line such that $vw \in E \IFFF I_v \cap I_w \neq \emptyset$.
 Interval graphs appear to have first been discussed by Hajos \cite{H}.
 Now classical and well-known characterizations of interval graphs were given by
 Lekkerkerker and Boland \cite{LB} in 1962 and Gilmore and Hoffman \cite{GH} in 1964.
 Extensive investigations and generalizations have since followed
 \cite{CJL,CH,Coz,EFKS,G,GJ-dm,GJ-jct,GJ,GJT,GM, GMT,GT,JLL,JMcMM,JMS, JMdm,JMcn,JM05,K,PT,R,S,T}.
 An interval graph is {\em proper} iff it has a representation in which
 no interval contains another. Roberts \cite{R} introduced proper interval graphs
 and characterized them as interval graphs containing no $K_{1,3}$.
 Proskurowski and Telle \cite{PT} generalized this to {\em $q$-proper interval graphs}, graphs having an interval
 representation in which no interval is properly contained in more than $q$ others.

 This paper will forbid containments in the opposite direction. A {\em $p$-improper interval graph}
 is one having an interval representation in which no interval contains more than $p$ other intervals.
 The key difference between these generalizations is that Proskurowski and Telle \cite{PT} forbid
 {\it supersets} whereas here {\it subsets} are forbidden.

 By a {\it $p$-improper representation} we mean an interval
 representation with no interval containing more than $p$ other intervals.
 Obviously, if $G$ has such a representation and $H$ is a subgraph of $G$,
 then deleting from a representation of $G$ those intervals
 which correspond to vertices not in $H$ yields a representation of $H$.
 This hereditary property guarantees that the class $\sI_p$ of $p$-improper interval graphs
 has a minimal forbidden subgraph characterization. The class of proper interval graphs
 (which coincides with the class of unit interval \cite{GT}) is thus the class $\sI_0$.

 The Lekkerkerker-Boland theorem \cite{LB} says that chordless cycles
 and asteroidal triples form a defining class of forbidden subgraphs
 for the class of interval graphs.  Thus we will be interested in
 finding minimal forbidden subgraphs within the class of interval graphs.
 Let $\sM_p$ denote the set of {\it minimal forbidden interval
 subgraphs (MFISG)} for the class $\sI_p$ of $p$-improper interval graphs.
 The {\em impropriety} $\imp(G)$ of $G$ is the smallest $p$ such that
 $G$ has a $p$-improper representation.  Unlike the case of $q$-proper interval graphs which have an
 essentially unique MFISG for each $q$, $p$-improper interval graphs show
 a great diversity of MFISGs, as we will see below.
 Fig. \ref{fig:main} shows a complete list of the MFISGs for the
 first class $\sI_1$ with $p=1$ \cite{BJ}.  These ten MFISGs show the breadth of possibilities right at the beginning.
 The star $K_{1,p+3}$ is easily seen to be a MFISG for $\sI_p$.  This is the easiest case.
 The next easiest case is the {\em balanced} case which includes three examples from Fig. \ref{fig:main}.
 We will give a formal definition of balanced here and give a complete description of all MFISGs in this case.


 \section{Weight and Balance in Interval Graphs}

 Throughout this section $G = (V,E)$ will denote a finite, connected, interval graph.
 First we establish the notation for the central ideas of the paper.
 Recall that a finite, simple graph $G=(V,E)$ is an {\em interval graph} iff there is
 an assignment $\a: v \toto I_v$ of vertices $v$ of $G$ to intervals $I_v$
 on the real line such that $vw \in E \IFFF I_v \cap I_w \neq \emptyset$.
 If a representation $\a$ has been given, $\ell_v$ and $r_v$ will
 denote the left and right endpoints, resp., of the interval $I_v$ representing $v$.
 The {\em support} of a set $W \subseteq V$ of vertices in a
 representation $\a$ is the union of all intervals $I_w$ where $w \in W$.
 The {\em impropriety} $\imp_{\a}(z)$ of a vertex $z$ of $G$ with
 respect to the representation $\a$ is the number of representing
 intervals which lie inside $I_z$ (not counting $I_z$ itself).
 The {\em impropriety} $\imp(\a)$ of the representation $\a$ is the
 {\bf maximum} of the improprieties $\imp_{\a}(z)$ over all vertices $z$ of $G$.
 The {\em impropriety} $\imp(G)$ of $G$ is the {\bf minimum} of $\imp(\a)$
 over all representations.  A representation which minimizes the
 impropriety will be called a {\em minimal representation}.
 That is, a representation $\a$ is minimal iff $\imp(\a) = \imp(G)$.

 For $z \in V$, a component of $G \setminus \{z\}$ will be called a
 {\em local component} at $z$ (or more simply, just a {\em component at} $z$).
 A local component is {\em exterior} iff it contains a vertex not adjacent to $z$.

 \begin{lemma}
 \label{lkkk}
 A vertex $z$ in an interval graph can have at most two exterior (local) components.
 \end{lemma}

 \begin{proof}
 If there are three exterior components $C_1, C_2, C_3$, choose vertices
 $a_1, a _2$, and $a_3$ at distance two to $z$ with $a_i \in C_i$.
 Then $a_1, a _2$, and $a_3$ form an asteroidal triple, which by
 \cite{LB} is forbidden in an interval graph.
 \end{proof}

 A vertex $z$ of $G$ is {\em type $k$} iff $v$ has exactly $k$ exterior components.
 By Lemma \ref{lkkk} $k$ can take on only three values: 0, 1, or 2.

 We now introduce a quantity which provides a lower bound on --- and sometimes an exact value for --- the impropriety.
 Suppose $z$ has $n$ local components $C_1, C_2, C_3, . . . , C_n$.
 The {\em weight} $\wt(z)$ of $z$ is the sum of the $n-2$
 smallest orders of the {\bf non-exterior} local components.
 The {\em weight} $\wt(G)$ of $G$ is the maximum of the weights of its vertices.
 Note that the weight is defined in terms of the graph $G$ directly
 and does not depend on any particular representation.  Impropriety,
 on the other hand, is defined in terms of representations of $G$.


 \begin{center}

   \begin{figure}
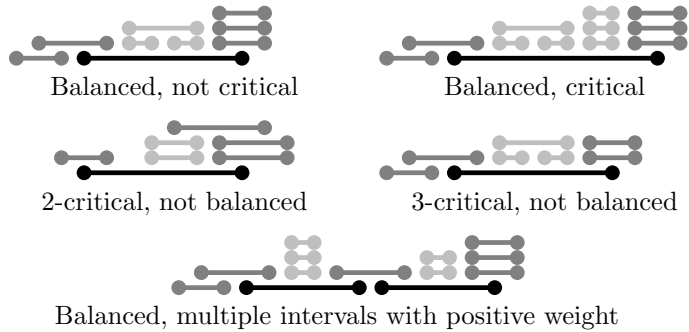

   \hspace{.35 in} 
   \begin{tabular}{c}

    \begin{tabular}{cc}
        \pspicture(\fifteen,\PicHeight)
            \rput(\seven,\LabelSpacing){Balanced, not critical}
            \psline[linewidth=\Thickness, linecolor=\SideColor]{\LINETYPE}(\zero,\Hzero)(\two,\Hzero)
            \psline[linewidth=\Thickness, linecolor=\SideColor]{\LINETYPE}(\one,\Hone)(\four,\Hone)
            \psline[linewidth=\Thickness, linecolor=\RootColor]{\LINETYPE}(\three,\Hzero)(\ten,\Hzero)
            \psline[linewidth=\Thickness, linecolor=\ClusterColor]{\LINETYPE}(\five,\Hone)(\six,\Hone)
            \psline[linewidth=\Thickness, linecolor=\ClusterColor]{\LINETYPE}(\seven,\Hone)(\eight,\Hone)
            \psline[linewidth=\Thickness, linecolor=\ClusterColor]{\LINETYPE}(\five,\Htwo)(\eight,\Htwo)
            \psline[linewidth=\Thickness, linecolor=\SideColor]{\LINETYPE}(\nine,\Hone)(\eleven,\Hone)
            \psline[linewidth=\Thickness, linecolor=\SideColor]{\LINETYPE}(\nine,\Htwo)(\eleven,\Htwo)
            \psline[linewidth=\Thickness, linecolor=\SideColor]{\LINETYPE}(\nine,\Hthree)(\eleven,\Hthree)
        \endpspicture

  &

        \pspicture(\thirteen,\PicHeight)
            \rput(\seven,\LabelSpacing){Balanced, critical}
            \psline[linewidth=\Thickness, linecolor=\SideColor]{\LINETYPE}(\zero,\Hzero)(\two,\Hzero)
            \psline[linewidth=\Thickness, linecolor=\SideColor]{\LINETYPE}(\one,\Hone)(\four,\Hone)
            \psline[linewidth=\Thickness, linecolor=\RootColor]{\LINETYPE}(\three,\Hzero)(\twelve,\Hzero)
            \psline[linewidth=\Thickness, linecolor=\ClusterColor]{\LINETYPE}(\five,\Hone)(\six,\Hone)
            \psline[linewidth=\Thickness, linecolor=\ClusterColor]{\LINETYPE}(\seven,\Hone)(\eight,\Hone)
            \psline[linewidth=\Thickness, linecolor=\ClusterColor]{\LINETYPE}(\five,\Htwo)(\eight,\Htwo)
            \psline[linewidth=\Thickness, linecolor=\ClusterColor]{\LINETYPE}(\nine,\Hone)(\ten,\Hone)
            \psline[linewidth=\Thickness, linecolor=\ClusterColor]{\LINETYPE}(\nine,\Htwo)(\ten,\Htwo)
            \psline[linewidth=\Thickness, linecolor=\ClusterColor]{\LINETYPE}(\nine,\Hthree)(\ten,\Hthree)
            \psline[linewidth=\Thickness, linecolor=\SideColor]{\LINETYPE}(\eleven,\Hone)(\thirteen,\Hone)
            \psline[linewidth=\Thickness, linecolor=\SideColor]{\LINETYPE}(\eleven,\Htwo)(\thirteen,\Htwo)
            \psline[linewidth=\Thickness, linecolor=\SideColor]{\LINETYPE}(\eleven,\Hthree)(\thirteen,\Hthree)
        \endpspicture

        \\
    \pspicture(\fifteen,\PicHeight)
    \rput(\seven,\LabelSpacing){2-critical, not balanced}
        \psline[linewidth=\Thickness, linecolor=\RootColor]{\LINETYPE}(\three,\Hzero)(\ten,\Hzero)
        \psline[linewidth=\Thickness, linecolor=\SideColor]{\LINETYPE}(\two,\Hone)(\four,\Hone)
        \psline[linewidth=\Thickness, linecolor=\ClusterColor]{\LINETYPE}(\six,\Hone)(\eight,\Hone)
        \psline[linewidth=\Thickness, linecolor=\ClusterColor]{\LINETYPE}(\six,\Htwo)(\eight,\Htwo)
        \psline[linewidth=\Thickness, linecolor=\SideColor]{\LINETYPE}(\seven,\Hthree)(\eleven,\Hthree)
        \psline[linewidth=\Thickness, linecolor=\SideColor]{\LINETYPE}(\nine,\Hone)(\twelve,\Hone)
        \psline[linewidth=\Thickness, linecolor=\SideColor]{\LINETYPE}(\nine,\Htwo)(\twelve,\Htwo)
    \endpspicture

   &

        \pspicture(\thirteen,\PicHeight)
            \rput(\seven,\LabelSpacing){3-critical, not balanced}
            \psline[linewidth=\Thickness, linecolor=\SideColor]{\LINETYPE}(\zero,\Hzero)(\two,\Hzero)
            \psline[linewidth=\Thickness, linecolor=\SideColor]{\LINETYPE}(\one,\Hone)(\four,\Hone)
            \psline[linewidth=\Thickness, linecolor=\RootColor]{\LINETYPE}(\three,\Hzero)(\ten,\Hzero)
            \psline[linewidth=\Thickness, linecolor=\ClusterColor]{\LINETYPE}(\five,\Hone)(\six,\Hone)
            \psline[linewidth=\Thickness, linecolor=\ClusterColor]{\LINETYPE}(\seven,\Hone)(\eight,\Hone)
            \psline[linewidth=\Thickness, linecolor=\ClusterColor]{\LINETYPE}(\five,\Htwo)(\eight,\Htwo)
            \psline[linewidth=\Thickness, linecolor=\SideColor]{\LINETYPE}(\nine,\Hone)(\eleven,\Hone)
            \psline[linewidth=\Thickness, linecolor=\SideColor]{\LINETYPE}(\nine,\Htwo)(\eleven,\Htwo)
        \endpspicture
        \\
        \end{tabular}

 \\
        \pspicture(\fifteen,\PicHeight)
            \rput(\seven,\LabelSpacing){Balanced, multiple intervals with positive weight}
            \psline[linewidth=\Thickness, linecolor=\SideColor]{\LINETYPE}(\zero,\Hzero)(\two,\Hzero)
            \psline[linewidth=\Thickness, linecolor=\SideColor]{\LINETYPE}(\one,\Hone)(\four,\Hone)
            \psline[linewidth=\Thickness, linecolor=\RootColor]{\LINETYPE}(\three,\Hzero)(\eight,\Hzero)
            \psline[linewidth=\Thickness, linecolor=\ClusterColor]{\LINETYPE}(\five,\Hone)(\six,\Hone)
            \psline[linewidth=\Thickness, linecolor=\ClusterColor]{\LINETYPE}(\five,\Htwo)(\six,\Htwo)
            \psline[linewidth=\Thickness, linecolor=\ClusterColor]{\LINETYPE}(\five,\Hthree)(\six,\Hthree)
            \psline[linewidth=\Thickness, linecolor=\SideColor]{\LINETYPE}(\seven,\Hone)(\ten,\Hone)
            \psline[linewidth=\Thickness, linecolor=\RootColor]{\LINETYPE}(\nine,\Hzero)(\fourteen,\Hzero)
            \psline[linewidth=\Thickness, linecolor=\ClusterColor]{\LINETYPE}(\eleven,\Hone)(\twelve,\Hone)
            \psline[linewidth=\Thickness, linecolor=\ClusterColor]{\LINETYPE}(\eleven,\Htwo)(\twelve,\Htwo)
            \psline[linewidth=\Thickness, linecolor=\SideColor]{\LINETYPE}(\thirteen,\Hone)(\fifteen,\Hone)
            \psline[linewidth=\Thickness, linecolor=\SideColor]{\LINETYPE}(\thirteen,\Htwo)(\fifteen,\Htwo)
            \psline[linewidth=\Thickness, linecolor=\SideColor]{\LINETYPE}(\thirteen,\Hthree)(\fifteen,\Hthree)
        \endpspicture

 \\

        \pspicture(\fifteen,.5)
        \endpspicture
    \end{tabular}
    \caption{Illustrations of Balance and Criticality}
    \label{fig:??}
  \end{figure}

   \end{center}


 Let us consider some examples of this somewhat confusing concept.
 Let $X_1$ and $X_2$ denote generic exterior components.
 Let $A,B,C,D,F$ denote local components with orders $A = 5, B = 5, C = 5, D = 4, F = 2$.

 \begin{center}
 \begin{tabular}{|l|l|l|l|}
 \hline
 Suppose the local          &Excluded   & The counted   &  Weight   \\
 components at $z$ are      &Loc Comp   & orders are    &           \\
 \hline
 $X_1,X_2,A,B$              & $X_1,X_2$ & 5+5           & 10        \\
 $X_1,X_2,C,F$              & $X_1,X_2$ & 5+2           & 7         \\
 $X_1,A,B$                  & $X_1,A$   & 5             & 5         \\
 $X_1,C,F$                  & $X_1,C$   & 2             & 2         \\
 $A,B,C,D,F$                & $A,B$     & 5+4+2         & 11        \\
 $C,D,F$                    & $C,D$     & 2             & 2         \\
 \hline
 \end{tabular}
 \end{center}

 \begin{center}
 \begin{tabular}{|l|c|l|l|}
 \hline
 Suppose the local          & Nr. of & $n-2$ smallest    &  Weight   \\
 components at $z$ are      & Comp. & Non-Exterior       &           \\
                            &       & Local Comp.        &           \\

 \hline
 $X_1,X_2,A,B,C,D,F$        & $n = 7$   & $A,B,C,D,F$   & 5+5+5+4+2 = 21        \\
 $X_1,X_2,C,F$              & $n = 4$   & $C,F$         & 5+2 = 7               \\
 $X_1,X_2$                  & $n = 2$   & none          & 0                     \\
 $X_1,A,B,C$                & $n = 4$   & $B,C$         & 5+5 = 10              \\
 $X_1,C,F$                  & $n = 3$   & $F$           & 2                     \\
 $A,B,C,D,F$                & $n = 5$   & $C,D,F$       & 5+4+2 = 11            \\
 $C,D,F$                    & $n = 3$   & $F$           & 2                     \\
 $C,F$                      & $n = 2$   & none          & 0                     \\
 $D$                        & $n = 1$   & none          & 0                     \\
 \hline
 \end{tabular}
 \end{center}

 \begin{thm}
 \label{wt}
 If $z$ is any vertex of an interval graph $G$, the impropriety of
 $G$ is at least the weight of $z$.
 \end{thm}

 \begin{proof}
 Consider any interval representation $\a: v \to I_v$ of $G$.
 The supports of the local components are themselves disjoint intervals
 which lie left to right along the line.  Say the local
 components in this ordering are $A_1, A_2, A_3, . . . , A_n$.
 Then the components $A_2, A_3, . . . , A_{n-1}$ must have supports entirely inside $I_z$.
 Thus each of these local components lies in the neighborhood of $z$.
 Hence if there are exterior components they must be $A_1$ or $A_n$, or both.
 In any case, the $n-2$ components $A_2, A_3, . . . , A_{n-1}$ are
 not exterior and thus the sum of their orders is at least $\wt(z)$.
 Thus we have shown that in any representation, $I_z$ contains at
 least $\wt(z)$ other intervals.  Thus the impropriety of $G$ is at
 least $\wt(z)$, as desired.
 \end{proof}

 \begin{cor}
 \label{wtcor}
 For any interval graph $G$, $\imp(G) \ge \wt(G)$.
 \end{cor}

 $G$ is {\em balanced} iff $\wt(G) = \imp(G)$.  If $G$ is balanced, a vertex $z$ such
 that $\wt(z) = \imp(G)$ is a {\em basepoint} of $G$.
 Equivalently, $z$ is a basepoint iff $G$ is balanced and $z$ has maximum weight.
 Notice that a basepoint must have at least three local components
 since a vertex with only one or two local components has weight 0.

    \begin{figure}
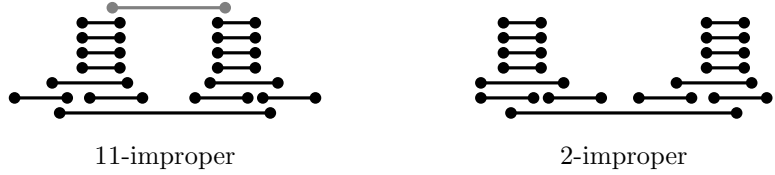

          \begin{center}
    \begin{tabular}{cc}
    \pspicture(7.1,2)
    \psline[linewidth=\ThicknessTwo, linecolor=\RootColor]{\LINETYPE}(1.1,0)(3.9,0)
    \rput(2.5,-.6){11-improper}
    \psline[linewidth=\ThicknessTwo, linecolor=\RootColor]{\LINETYPE}(.5,.2)(1.2,.2)
    \psline[linewidth=\ThicknessTwo, linecolor=\RootColor]{\LINETYPE}(1.5, .2)(2.2,.2)
    \psline[linewidth=\ThicknessTwo, linecolor=\RootColor]{\LINETYPE}(1,.4)(2,.4)
    \psline[linewidth=\ThicknessTwo, linecolor=\RootColor]{\LINETYPE}(1.4,.6)(1.9,.6)
    \psline[linewidth=\ThicknessTwo, linecolor=\RootColor]{\LINETYPE}(1.4,.8)(1.9,.8)
    \psline[linewidth=\ThicknessTwo, linecolor=\RootColor]{\LINETYPE}(1.4,1)(1.9,1)
    \psline[linewidth=\ThicknessTwo, linecolor=\RootColor]{\LINETYPE}(1.4,1.2)(1.9,1.2)

    \psline[linewidth=\ThicknessTwo, linecolor=\RootColor]{\LINETYPE}(3.8,.2)(4.5,.2)
    \psline[linewidth=\ThicknessTwo, linecolor=\RootColor]{\LINETYPE}(2.9, .2)(3.6,.2)
    \psline[linewidth=\ThicknessTwo, linecolor=\RootColor]{\LINETYPE}(3.1,.4)(4,.4)
    \psline[linewidth=\ThicknessTwo, linecolor=\RootColor]{\LINETYPE}(3.2,.6)(3.7,.6)
    \psline[linewidth=\ThicknessTwo, linecolor=\RootColor]{\LINETYPE}(3.2,.8)(3.7,.8)
    \psline[linewidth=\ThicknessTwo, linecolor=\RootColor]{\LINETYPE}(3.2,1)(3.7,1)
    \psline[linewidth=\ThicknessTwo, linecolor=\RootColor]{\LINETYPE}(3.2,1.2)(3.7,1.2)

    \psline[linewidth=\ThicknessTwo, linecolor=\SideColor]{\LINETYPE}(1.8,1.4)(3.3,1.4)

    \endpspicture

    \pspicture(4,2)
    \psline[linewidth=\ThicknessTwo, linecolor=\RootColor]{\LINETYPE}(0,0)(3,0)
    \rput(1.5,-.6){2-improper}

    \psline[linewidth=\ThicknessTwo, linecolor=\RootColor]{\LINETYPE}(1.7,.2)(2.4,.2)
    \psline[linewidth=\ThicknessTwo, linecolor=\RootColor]{\LINETYPE}(2.7, .2)(3.4,.2)
    \psline[linewidth=\ThicknessTwo, linecolor=\RootColor]{\LINETYPE}(2.2,.4)(3.2,.4)
    \psline[linewidth=\ThicknessTwo, linecolor=\RootColor]{\LINETYPE}(2.6,.6)(3.1,.6)
    \psline[linewidth=\ThicknessTwo, linecolor=\RootColor]{\LINETYPE}(2.6,.8)(3.1,.8)
    \psline[linewidth=\ThicknessTwo, linecolor=\RootColor]{\LINETYPE}(2.6,1)(3.1,1)
    \psline[linewidth=\ThicknessTwo, linecolor=\RootColor]{\LINETYPE}(2.6,1.2)(3.1,1.2)

    \psline[linewidth=\ThicknessTwo, linecolor=\RootColor]{\LINETYPE}(.5,.2)(1.2,.2)
    \psline[linewidth=\ThicknessTwo, linecolor=\RootColor]{\LINETYPE}(-.4, .2)(.3,.2)
    \psline[linewidth=\ThicknessTwo, linecolor=\RootColor]{\LINETYPE}(-.4,.4)(.7,.4)
    \psline[linewidth=\ThicknessTwo, linecolor=\RootColor]{\LINETYPE}(-.1,.6)(.4,.6)
    \psline[linewidth=\ThicknessTwo, linecolor=\RootColor]{\LINETYPE}(-.1,.8)(.4,.8)
    \psline[linewidth=\ThicknessTwo, linecolor=\RootColor]{\LINETYPE}(-.1,1)(.4,1)
    \psline[linewidth=\ThicknessTwo, linecolor=\RootColor]{\LINETYPE}(-.1,1.2)(.4,1.2)

    \endpspicture

    \\
    \pspicture(2,.4)
    \endpspicture
    & 
    \pspicture(2,.4)
    \endpspicture

    \end{tabular}
        \end{center}
 \caption{Example of Instability in the Impropriety}
 \label{unstab}
    \end{figure}


 \begin{thm}
 \label{wt0}
 If $G$ is a connected, interval graph, then the vertices of
 positive weight induce a disjoint union of paths.
 \end{thm}

 \begin{proof}
 Suppose $\a$ is a representation of $G$ and suppose $I_v \subseteq I_w$.
 Then every neighbor of $v$ is also a neighbor of $w$. Hence in $G \setminus \{v\}$,
 all the neighbors of $v$ are still connected.  Hence $v$ has only
 one local component, so $\wt(v) = 0$.

 Now suppose some vertex $v$ with $\wt(v) > 0$ has three neighbors $a,b,c$ also with positive weight.
 Suppose $\a$ is a representation of $G$.  We saw above that none of the intervals $I_a$, $I_b$, or $I_c$
 can be contained in $I_v$.  Thus two of these intervals must exit $I_v$ on the same side.
 Say, $I_a$ and $I_b$ exits $I_v$ through the right end point $r_v$ of $I_v$.
 Without loss of generality, assume $\ell_a \le \ell_b$. Since no
 interval of positive weight can contain another, $r_a \le r_b$ is forced.
 Thus $I_a \subseteq I_v \cup I_b$.  But this means that any
 neighbor of $a$ must be a neighbor of either $v$ or $b$.  Since $v$ and $b$ are adjacent,
 it follows that $a$ has only one local component, and hence has $\wt(a) = 0$, a contradiction.

 \bigskip

 \end{proof}



   \begin{figure}
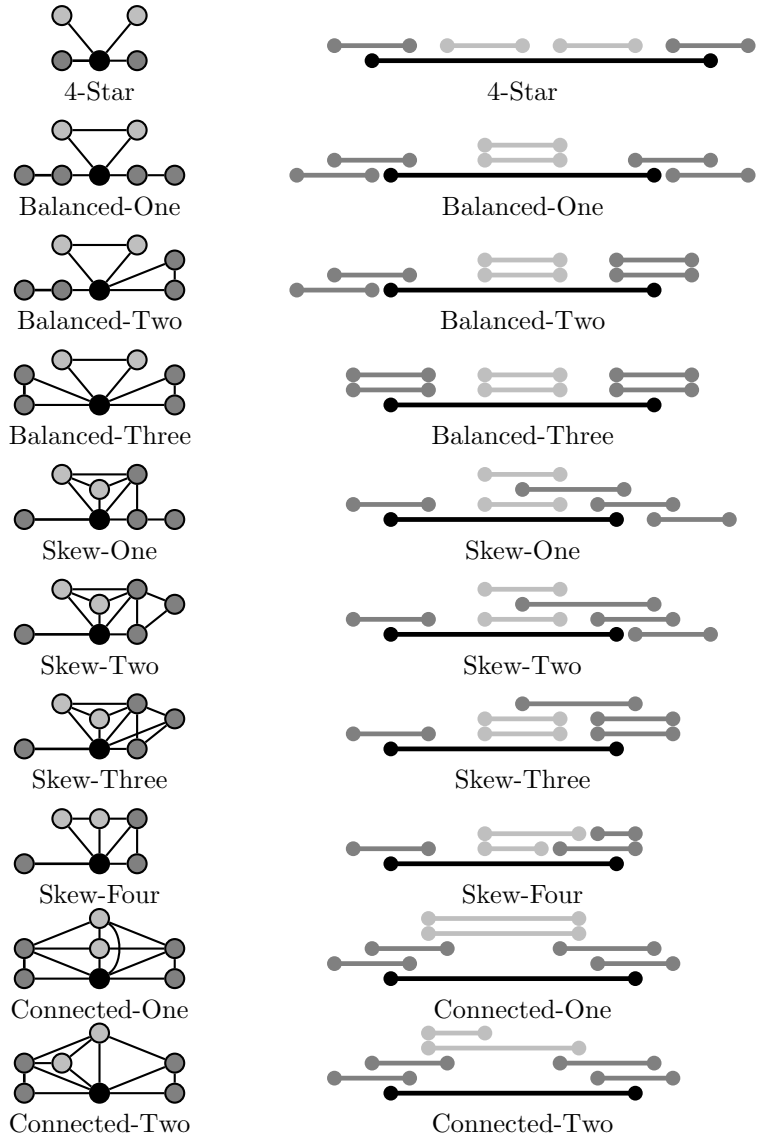

   \hspace{.1 in}
        \begin{tabular}{cc}

        \pspicture(\PicWidthL,\PicHeight)
        \rput(\Wmid,\LabelSpacing){4-Star}
            \cnode[fillstyle=solid, fillcolor=\SideColor](\Wbone,0){\R}{n1}
            \cnode[fillstyle=solid, fillcolor=black](\Wmid,0){\R}{n2}
            \cnode[fillstyle=solid, fillcolor=\SideColor](\Wpone,0){\R}{n3}
            \cnode[fillstyle=solid, fillcolor=light\SideColor](\Wbone,\Hthree){\R}{n5}
            \cnode[fillstyle=solid, fillcolor=light\SideColor](\Wpone,\Hthree){\R}{n6}

            \ncline[linewidth=1pt]{n1}{n2}
            \ncline[]{n2}{n3}
            \ncline[]{n2}{n5}
            \ncline[]{n2}{n6}
        \endpspicture

&

        \pspicture(\PicWidth,\PicHeight)
        \rput(3,\LabelSpacing){4-Star}
            \psline[linewidth=\Thickness, linecolor=\RootColor]{\LINETYPE}(1,0)(5.5,0)
            \psline[linewidth=\Thickness, linecolor=\SideColor]{\LINETYPE}(0.5,\Hone)(1.5,\Hone)
            \psline[linewidth=\Thickness, linecolor=light\SideColor]{\LINETYPE}(2,\Hone)(3,\Hone)
            \psline[linewidth=\Thickness, linecolor=light\SideColor]{\LINETYPE}(3.5,\Hone)(4.5,\Hone)
            \psline[linewidth=\Thickness, linecolor=\SideColor]{\LINETYPE}(5,\Hone)(6,\Hone)
        \endpspicture

\\

\pspicture(\PicWidthL,\PicHeight)
\rput(\Wmid,\LabelSpacing){Balanced-One} \cnode[fillstyle=solid,
fillcolor=\SideColor](\Wbtwo,0){\R}{n0} \cnode[fillstyle=solid,
fillcolor=\SideColor](\Wbone,0){\R}{n1} \cnode[fillstyle=solid,
fillcolor=\RootColor](\Wmid,0){\R}{n2} \cnode[fillstyle=solid,
fillcolor=\SideColor](\Wpone,0){\R}{n3} \cnode[fillstyle=solid,
fillcolor=\SideColor](\Wptwo,0){\R}{n4} \cnode[fillstyle=solid,
fillcolor=light\SideColor](\Wbone,\Hthree){\R}{n5}
\cnode[fillstyle=solid,
fillcolor=light\SideColor](\Wpone,\Hthree){\R}{n6}

\ncline[linewidth=1pt]{n0}{n1} \ncline[]{n1}{n2} \ncline[]{n2}{n3}
\ncline[]{n3}{n4} \ncline[]{n2}{n5} \ncline[]{n2}{n6}
\ncline[]{n5}{n6}
\endpspicture

&

    \pspicture(\PicWidth,\PicHeight)
    \rput(3,\LabelSpacing){Balanced-One}
    \psline[linewidth=\Thickness, linecolor=\RootColor]{\LINETYPE}(1.25,\Hzero)(4.75,\Hzero)
    \psline[linewidth=\Thickness, linecolor=\SideColor]{\LINETYPE}(0,\Hzero)(1,\Hzero)
    \psline[linewidth=\Thickness, linecolor=\SideColor]{\LINETYPE}(.5,\Hone)(1.5,\Hone)
    \psline[linewidth=\Thickness, linecolor=light\SideColor]{\LINETYPE}(2.5,\Hone)(3.5,\Hone)
    \psline[linewidth=\Thickness, linecolor=light\SideColor]{\LINETYPE}(2.5,\Htwo)(3.5,\Htwo)
    \psline[linewidth=\Thickness, linecolor=\SideColor]{\LINETYPE}(4.5,\Hone)(5.5,\Hone)
    \psline[linewidth=\Thickness, linecolor=\SideColor]{\LINETYPE}(5,\Hzero)(6,\Hzero)

    \endpspicture

\\

\pspicture(\PicWidthL,\PicHeight)
\rput(\Wmid,\LabelSpacing){Balanced-Two} \cnode[fillstyle=solid,
fillcolor=\SideColor](\Wbtwo,0){\R}{leftendcap1}
\cnode[fillstyle=solid,
fillcolor=\SideColor](\Wbone,0){\R}{leftendcap2}
\cnode[fillstyle=solid, fillcolor=\RootColor](\Wmid,0){\R}{middle}
\cnode[fillstyle=solid,
fillcolor=\SideColor](\Wptwo,0){\R}{rightk2a}
\cnode[fillstyle=solid,
fillcolor=\SideColor](\Wptwo,\Htwo){\R}{rightk2b}
\cnode[fillstyle=solid,
fillcolor=light\SideColor](\Wpone,\Hthree){\R}{clustera}
\cnode[fillstyle=solid,
fillcolor=light\SideColor](\Wbone,\Hthree){\R}{clusterb}

\ncline[linewidth=1pt]{leftendcap1}{leftendcap2}
\ncline[]{leftendcap2}{middle} \ncline[]{middle}{rightk2a}
\ncline[]{middle}{rightk2b} \ncline[]{middle}{clustera}
\ncline[]{middle}{clusterb} \ncline[]{clustera}{clusterb}
\ncline[]{rightk2a}{rightk2b}
\endpspicture

&

    \pspicture(\PicWidth,\PicHeight)
    \rput(3,\LabelSpacing){Balanced-Two}
    \psline[linewidth=\Thickness, linecolor=\RootColor]{\LINETYPE}(1.25,\Hzero)(4.75,\Hzero)
    \psline[linewidth=\Thickness, linecolor=\SideColor]{\LINETYPE}(0,\Hzero)(1,\Hzero)
    \psline[linewidth=\Thickness, linecolor=\SideColor]{\LINETYPE}(.5,\Hone)(1.5,\Hone)
    \psline[linewidth=\Thickness, linecolor=light\SideColor]{\LINETYPE}(2.5,\Hone)(3.5,\Hone)
    \psline[linewidth=\Thickness, linecolor=light\SideColor]{\LINETYPE}(2.5,\Htwo)(3.5,\Htwo)
    \psline[linewidth=\Thickness, linecolor=\SideColor]{\LINETYPE}(4.25,\Hone)(5.25,\Hone)
    \psline[linewidth=\Thickness, linecolor=\SideColor]{\LINETYPE}(4.25,\Htwo)(5.25,\Htwo)

    \endpspicture

\\

    \pspicture(\PicWidthL,\PicHeight)
    \rput(\Wmid,\LabelSpacing){Balanced-Three}
        \cnode[fillstyle=solid, fillcolor=\SideColor](\Wbtwo,0){\R}{leftk2a}
        \cnode[fillstyle=solid, fillcolor=\SideColor](\Wbtwo,\Htwo){\R}{leftk2b}
        \cnode[fillstyle=solid, fillcolor=\RootColor](\Wmid,0){\R}{middle}
        \cnode[fillstyle=solid, fillcolor=\SideColor](\Wptwo,0){\R}{rightk2a}
        \cnode[fillstyle=solid, fillcolor=\SideColor](\Wptwo,\Htwo){\R}{rightk2b}
        \cnode[fillstyle=solid, fillcolor=light\SideColor](\Wbone,\Hthree){\R}{clustera}
        \cnode[fillstyle=solid, fillcolor=light\SideColor](\Wpone,\Hthree){\R}{clusterb}

        \ncline[linewidth=1pt]{leftk2a}{leftk2b}
        \ncline[]{leftk2a}{middle}
        \ncline[]{leftk2b}{middle}
        \ncline[]{middle}{rightk2a}
        \ncline[]{middle}{rightk2b}
        \ncline[]{middle}{clustera}
        \ncline[]{middle}{clusterb}
        \ncline[]{clustera}{clusterb}
        \ncline[]{leftk2a}{leftk2b}
        \ncline[]{rightk2a}{rightk2b}
    \endpspicture

&

    \pspicture(\PicWidth,\PicHeight)
    \rput(3,\LabelSpacing){Balanced-Three}
        \psline[linewidth=\Thickness, linecolor=\RootColor]{\LINETYPE}(1.25,\Hzero)(4.75,\Hzero)
        \psline[linewidth=\Thickness, linecolor=\SideColor]{\LINETYPE}(.75,\Htwo)(1.75,\Htwo)
        \psline[linewidth=\Thickness, linecolor=\SideColor]{\LINETYPE}(.75,\Hone)(1.75,\Hone)
        \psline[linewidth=\Thickness, linecolor=\ClusterColor]{\LINETYPE}(2.5,\Hone)(3.5,\Hone)
        \psline[linewidth=\Thickness, linecolor=\ClusterColor]{\LINETYPE}(2.5,\Htwo)(3.5,\Htwo)
        \psline[linewidth=\Thickness, linecolor=\SideColor]{\LINETYPE}(4.25,\Hone)(5.25,\Hone)
        \psline[linewidth=\Thickness, linecolor=\SideColor]{\LINETYPE}(4.25,\Htwo)(5.25,\Htwo)
    \endpspicture

\\

    \pspicture(\PicWidthL,\PicHeight)
    \rput(\Wmid,\LabelSpacing){Skew-One}
        \cnode[fillstyle=solid, fillcolor=\SideColor](\Wbtwo,0){\R}{leftk1}
        \cnode[fillstyle=solid, fillcolor=\RootColor](\Wmid,0){\R}{root}
        \cnode[fillstyle=solid, fillcolor=\SideColor](\Wpone,0){\R}{rightcap1}
        \cnode[fillstyle=solid, fillcolor=\SideColor](\Wptwo,0){\R}{rightcap2}
        \cnode[fillstyle=solid, fillcolor=\SideColor](\Wpone,\Hthree){\R}{rightother}
        \cnode[fillstyle=solid, fillcolor=\ClusterColor](\Wbone,\Hthree){\R}{clustera}
        \cnode[fillstyle=solid, fillcolor=\ClusterColor](\Wmid,\Htwo){\R}{clusterb}

        \ncline[linewidth=1pt]{leftk1}{root}
        \ncline[]{root}{rightcap1}
        \ncline[]{root}{rightother}
        \ncline[]{root}{clustera}
        \ncline[]{root}{clusterb}
        \ncline[]{clustera}{clusterb}
        \ncline[]{rightcap1}{rightcap2}
        \ncline[]{rightcap1}{rightother}
        \ncline[]{rightother}{clustera}
        \ncline[]{rightother}{clusterb}
    \endpspicture

&

    \pspicture(\PicWidth,\PicHeight)
    \rput(3,\LabelSpacing){Skew-One}
        \psline[linewidth=\Thickness, linecolor=\RootColor]{\LINETYPE}(1.25,\Hzero)(4.25,\Hzero)
        \psline[linewidth=\Thickness, linecolor=\SideColor]{\LINETYPE}(.75,\Hone)(1.75,\Hone)
        \psline[linewidth=\Thickness, linecolor=\ClusterColor]{\LINETYPE}(2.5,\Hthree)(3.5,\Hthree)
        \psline[linewidth=\Thickness, linecolor=\ClusterColor]{\LINETYPE}(2.5,\Hone)(3.5,\Hone)
        \psline[linewidth=\Thickness, linecolor=\SideColor]{\LINETYPE}(3,\Htwo)(4.35,\Htwo)
        \psline[linewidth=\Thickness, linecolor=\SideColor]{\LINETYPE}(4.75,\Hzero)(5.75,\Hzero)
        \psline[linewidth=\Thickness, linecolor=\SideColor]{\LINETYPE}(4,\Hone)(5,\Hone)
    \endpspicture

        \\

    \pspicture(\PicWidthL,\PicHeight)
    \rput(\Wmid,\LabelSpacing){Skew-Two}
        \cnode[fillstyle=solid, fillcolor=\SideColor](\Wbtwo,0){\R}{leftk1}
        \cnode[fillstyle=solid, fillcolor=\RootColor](\Wmid,0){\R}{root}
        \cnode[fillstyle=solid, fillcolor=\SideColor](\Wpone,0){\R}{rightk3a}
        \cnode[fillstyle=solid, fillcolor=\SideColor](\Wptwo,\Htwo){\R}{rightk3b}
        \cnode[fillstyle=solid, fillcolor=\SideColor](\Wpone,\Hthree){\R}{rightk3c}
        \cnode[fillstyle=solid, fillcolor=\ClusterColor](\Wbone,\Hthree){\R}{clustera}
        \cnode[fillstyle=solid, fillcolor=\ClusterColor](\Wmid,\Htwo){\R}{clusterb}

        \ncline[linewidth=1pt]{leftk1}{root}
        \ncline[]{root}{rightk3a}
        \ncline[]{root}{rightk3c}
        \ncline[]{root}{clustera}
        \ncline[]{root}{clusterb}
        \ncline[]{clustera}{clusterb}
        \ncline[]{rightk3a}{rightk3b}
        \ncline[]{rightk3b}{rightk3c}
        \ncline[]{rightk3a}{rightk3c}
        \ncline[]{clustera}{rightk3c}
        \ncline[]{clusterb}{rightk3c}
    \endpspicture

&

    \pspicture(\PicWidth,\PicHeight)
    \rput(3,\LabelSpacing){Skew-Two}
        \psline[linewidth=\Thickness, linecolor=\RootColor]{\LINETYPE}(1.25,\Hzero)(4.25,\Hzero)
        \psline[linewidth=\Thickness, linecolor=\SideColor]{\LINETYPE}(.75,\Hone)(1.75,\Hone)
        \psline[linewidth=\Thickness, linecolor=\ClusterColor]{\LINETYPE}(2.5,\Hthree)(3.5,\Hthree)
        \psline[linewidth=\Thickness, linecolor=\ClusterColor]{\LINETYPE}(2.5,\Hone)(3.5,\Hone)
        \psline[linewidth=\Thickness, linecolor=\SideColor]{\LINETYPE}(3,\Htwo)(4.75,\Htwo)
        \psline[linewidth=\Thickness, linecolor=\SideColor]{\LINETYPE}(4.5,\Hzero)(5.5,\Hzero)
        \psline[linewidth=\Thickness, linecolor=\SideColor]{\LINETYPE}(4,\Hone)(5,\Hone)
    \endpspicture

    \\

        \pspicture(\PicWidthL,\PicHeight)
    \rput(\Wmid,\LabelSpacing){Skew-Three}
        \cnode[fillstyle=solid, fillcolor=\SideColor](\Wbtwo,0){\R}{leftk1}
        \cnode[fillstyle=solid, fillcolor=\RootColor](\Wmid,0){\R}{root}
        \cnode[fillstyle=solid, fillcolor=\SideColor](\Wpone,0){\R}{rightk3a}
        \cnode[fillstyle=solid, fillcolor=\SideColor](\Wptwo,\Htwo){\R}{rightk3b}
        \cnode[fillstyle=solid, fillcolor=\SideColor](\Wpone,\Hthree){\R}{rightk3c}
        \cnode[fillstyle=solid, fillcolor=\ClusterColor](\Wbone,\Hthree){\R}{clustera}
        \cnode[fillstyle=solid, fillcolor=\ClusterColor](\Wmid,\Htwo){\R}{clusterb}

        \ncline[linewidth=1pt]{leftk1}{root}
        \ncline[]{root}{rightk3a}
        \ncline[]{root}{rightk3b}
        \ncline[]{root}{rightk3c}
        \ncline[]{root}{clustera}
        \ncline[]{root}{clusterb}
        \ncline[]{clustera}{clusterb}
        \ncline[]{rightk3a}{rightk3b}
        \ncline[]{rightk3b}{rightk3c}
        \ncline[]{rightk3a}{rightk3c}
        \ncline[]{clustera}{rightk3c}
        \ncline[]{clusterb}{rightk3c}
    \endpspicture

&

    \pspicture(\PicWidth,\PicHeight)
    \rput(3,\LabelSpacing){Skew-Three}
        \psline[linewidth=\Thickness, linecolor=\RootColor]{\LINETYPE}(1.25,\Hzero)(4.25,\Hzero)
        \psline[linewidth=\Thickness, linecolor=\SideColor]{\LINETYPE}(.75,\Hone)(1.75,\Hone)
        \psline[linewidth=\Thickness, linecolor=\ClusterColor]{\LINETYPE}(2.5,\Hone)(3.5,\Hone)
        \psline[linewidth=\Thickness, linecolor=\ClusterColor]{\LINETYPE}(2.5,\Htwo)(3.5,\Htwo)
        \psline[linewidth=\Thickness, linecolor=\SideColor]{\LINETYPE}(3,\Hthree)(4.5,\Hthree)
        \psline[linewidth=\Thickness, linecolor=\SideColor]{\LINETYPE}(4,\Hone)(5,\Hone)
        \psline[linewidth=\Thickness, linecolor=\SideColor]{\LINETYPE}(4,\Htwo)(5,\Htwo)
    \endpspicture

    \\

        \pspicture(\PicWidthL,\PicHeight)
    \rput(\Wmid,\LabelSpacing){Skew-Four}
        \cnode[fillstyle=solid, fillcolor=\SideColor](\Wbtwo,0){\R}{leftk1}
        \cnode[fillstyle=solid, fillcolor=\RootColor](\Wmid,0){\R}{root}
        \cnode[fillstyle=solid, fillcolor=\SideColor](\Wpone,0){\R}{rightk2a}
        \cnode[fillstyle=solid, fillcolor=\SideColor](\Wpone,\Hthree){\R}{rightk2b}
        \cnode[fillstyle=solid, fillcolor=\ClusterColor](\Wbone,\Hthree){\R}{clustera}
        \cnode[fillstyle=solid, fillcolor=\ClusterColor](\Wmid,\Hthree){\R}{clusterb}

        \ncline[linewidth=1pt]{leftk1}{root}
        \ncline[]{root}{rightk2a}
        \ncline[]{root}{rightk2b}
        \ncline[]{root}{clustera}
        \ncline[]{root}{clusterb}
        \ncline[]{clustera}{clusterb}
        \ncline[]{rightk2a}{rightk2b}
        \ncline[]{clusterb}{rightk2b}
    \endpspicture

&

    \pspicture(\PicWidth,\PicHeight)
    \rput(3,\LabelSpacing){Skew-Four}
        \psline[linewidth=\Thickness, linecolor=\RootColor]{\LINETYPE}(1.25,\Hzero)(4.25,\Hzero)
        \psline[linewidth=\Thickness, linecolor=\SideColor]{\LINETYPE}(.75,\Hone)(1.75,\Hone)
        \psline[linewidth=\Thickness, linecolor=\ClusterColor]{\LINETYPE}(2.5,\Hone)(3.25,\Hone)
        \psline[linewidth=\Thickness, linecolor=\ClusterColor]{\LINETYPE}(2.5,\Htwo)(3.75,\Htwo)
        \psline[linewidth=\Thickness, linecolor=\SideColor]{\LINETYPE}(3.5,\Hone)(4.5,\Hone)
        \psline[linewidth=\Thickness, linecolor=\SideColor]{\LINETYPE}(4,\Htwo)(4.5,\Htwo)
    \endpspicture

\\

        \pspicture(\PicWidthL,\PicHeight)
    \rput(\Wmid,\LabelSpacing){Connected-One}
        \cnode[fillstyle=solid, fillcolor=\SideColor](\Wbtwo,0){\R}{leftk2a}
        \cnode[fillstyle=solid, fillcolor=\SideColor](\Wbtwo,\Htwo){\R}{leftk2b}
        \cnode[fillstyle=solid, fillcolor=\RootColor](\Wmid,0){\R}{middle}
        \cnode[fillstyle=solid, fillcolor=\SideColor](\Wptwo,0){\R}{rightk2a}
        \cnode[fillstyle=solid, fillcolor=\SideColor](\Wptwo,\Htwo){\R}{rightk2b}
        \cnode[fillstyle=solid, fillcolor=light\SideColor](\Wmid,\Hfour){\R}{cluster1}
        \cnode[fillstyle=solid, fillcolor=light\SideColor](\Wmid,\Htwo){\R}{cluster2}

        \ncline[linewidth=1pt]{leftk2a}{leftk2b}
        \ncline[]{rightk2a}{rightk2b}
        \ncline[]{middle}{leftk2b}
        \ncline[]{middle}{leftk2a}
        \ncline[]{middle}{rightk2a}
        \ncline[]{middle}{rightk2b}
        \ncline[]{middle}{cluster2}
        \ncline[]{leftk2b}{cluster2}
        \ncline[]{rightk2b}{cluster2}
        \ncline[]{cluster1}{cluster2}
        \ncline[]{cluster1}{leftk2b}
        \ncline[]{cluster1}{rightk2b}
        \ncarc[arcangleA=60, arcangleB=60]{cluster1}{middle}
    \endpspicture

&

    \pspicture(\PicWidth,\PicHeight)
    \rput(3,\LabelSpacing){Connected-One}
        \psline[linewidth=\Thickness, linecolor=\RootColor]{\LINETYPE}(1.25,0)(4.5,0)
        \psline[linewidth=\Thickness, linecolor=\SideColor]{\LINETYPE}(0.5,\Hone)(1.5,\Hone)
        \psline[linewidth=\Thickness, linecolor=\SideColor]{\LINETYPE}(1,\Htwo)(2,\Htwo)
        \psline[linewidth=\Thickness, linecolor=light\SideColor]{\LINETYPE}(1.75,\Hthree)(3.75,\Hthree)
        \psline[linewidth=\Thickness, linecolor=light\SideColor]{\LINETYPE}(1.75,\Hfour)(3.75,\Hfour)
        \psline[linewidth=\Thickness, linecolor=\SideColor]{\LINETYPE}(3.5,\Htwo)(4.75,\Htwo)
        \psline[linewidth=\Thickness, linecolor=\SideColor]{\LINETYPE}(4,\Hone)(5,\Hone)

    \endpspicture

\\

    \pspicture(\PicWidthL,\PicHeight)
    \rput(\Wmid,\LabelSpacing){Connected-Two}
        \cnode[fillstyle=solid, fillcolor=\SideColor](\Wbtwo,0){\R}{leftk2a}
        \cnode[fillstyle=solid, fillcolor=\SideColor](\Wbtwo,\Htwo){\R}{leftk2b}
        \cnode[fillstyle=solid, fillcolor=\RootColor](\Wmid,0){\R}{middle}
        \cnode[fillstyle=solid, fillcolor=\SideColor](\Wptwo,0){\R}{rightk2a}
        \cnode[fillstyle=solid, fillcolor=\SideColor](\Wptwo,\Htwo){\R}{rightk2b}
        \cnode[fillstyle=solid, fillcolor=light\SideColor](\Wbone,\Htwo){\R}{cluster1}
        \cnode[fillstyle=solid, fillcolor=light\SideColor](\Wmid,\Hfour){\R}{cluster2}

        \ncline[linewidth=1pt]{leftk2a}{leftk2b}
        \ncline[]{rightk2a}{rightk2b}
        \ncline[]{middle}{leftk2b}
        \ncline[]{middle}{leftk2a}
        \ncline[]{middle}{rightk2a}
        \ncline[]{middle}{rightk2b}
        \ncline[]{middle}{cluster1}
        \ncline[]{middle}{cluster2}
        \ncline[]{leftk2b}{cluster2}
        \ncline[]{rightk2b}{cluster2}
        \ncline[]{cluster1}{cluster2}
        \ncline[]{cluster1}{leftk2b}
    \endpspicture

&

    \pspicture(\PicWidth,\PicHeight)
    \rput(3,\LabelSpacing){Connected-Two}
        \psline[linewidth=\Thickness, linecolor=\RootColor]{\LINETYPE}(1.25,0)(4.5,0)
        \psline[linewidth=\Thickness, linecolor=\SideColor]{\LINETYPE}(0.5,\Hone)(1.5,\Hone)
        \psline[linewidth=\Thickness, linecolor=\SideColor]{\LINETYPE}(1,\Htwo)(2,\Htwo)
        \psline[linewidth=\Thickness, linecolor=light\SideColor]{\LINETYPE}(1.75,\Hthree)(3.75,\Hthree)
        \psline[linewidth=\Thickness, linecolor=light\SideColor]{\LINETYPE}(1.75,\Hfour)(2.5,\Hfour)
        \psline[linewidth=\Thickness, linecolor=\SideColor]{\LINETYPE}(3.5,\Htwo)(4.75,\Htwo)
        \psline[linewidth=\Thickness, linecolor=\SideColor]{\LINETYPE}(4,\Hone)(5,\Hone)
    \endpspicture

\\
        \pspicture(\PicWidthL,.4)
        \endpspicture
    &
        \pspicture(\PicWidth,.4)
        \endpspicture

\end{tabular}

    \caption{Minimal Forbidden Subgraphs for $\sI_1$}
    \label{fig:main}

 \end{figure}

 \section{$p$-critical Interval Graphs}

 An interval graph $G$ is {\em $p$-critical} with respect to
 impropriety iff $G$ has impropriety $p$ but every proper induced
 subgraph of $G$ has impropriety strictly less than $p$.  Note that
 the concept of $p$-critical only makes sense for $p > 0$.
 Clearly, a $p+1$-critical graph is a MFISG for the class $\sI_p$ of
 $p$-improper interval graphs.  The converse is not so clear.
 Fig. \ref{unstab} gives an example where the impropriety changes
 drastically with the removal of a single vertex.

 \begin{thm}
 \label{2pts}
 Let $z$ be a vertex of maximum weight in a balanced $p$-critical graph $G$.
 If $C$ is an exterior local component at $z$, then $C$ consists of exactly two vertices.
 \end{thm}

 \begin{proof}
 Let $v$ be a vertex in $C$ at distance 2 from $z$, and let $w$ be a common neighbor of $v$ and $z$.
 Let $H$ be the graph obtained from $G$ by deleting all vertices of $C$ other than $v$ and $w$.
 The local components at $z$ in $H$ are the same as in $G$ except
 that $C$ is replaced by $\{v,w\}$.  Hence the $n-2$ smallest non-exterior local components
 at $z$ in $H$ are the same as in $G$.  Thus the weight of $z$ in
 $H$ is the same as the weight of $z$ in $G$.  Since $G$ is
 balanced and $C$ contains vertices other than $v$ and $w$, then $H$ is a proper induced subgraph
 of $G$ and hence has a strictly smaller impropriety.  Thus we have
    \[
        \wt_H(z) \; \le \; \imp(H) \; < \; \imp(G) \; = \;  \wt_G(z) \; = \; \wt_H(z),
    \]
 a contradiction.  Hence $C$ must be just $\{v, w\}$ as desired.
 \end{proof}

 \begin{thm}
 \label{unibase}
 If $G$ is balanced and $p$-critical, then $G$ has exactly one basepoint.
 \end{thm}

 \begin{proof}
    Suppose $y$ and $z$ are distinct basepoints.  Because $G$ is
    connected, $y$ must belong to some local component $C$ of $z$. This
    component must also contain all $p$ of the vertices whose intervals are contained in $I_y$.
    Since $G$ is balanced and $p \ge 1$, any basepoint for $G$ must have at least three local components
    and hence at least three neighbors.
    Thus since exterior components contain only 2 vertices by Lemma \ref{2pts},
    $C$ cannot be exterior.  Dually, $z$ is contained in a local
    component $D$ at $y$, which, dually, is not exterior.  Since $z$ has at least three local
    components, there is a local component $A$ at $z$ which is
    disjoint from $C$.  That is, $z$ is adjacent to vertices not adjacent to $y$.
    But that means, $D$ is an exterior component at $y$, a contradiction.
 \end{proof}

 \begin{thm}
 \label{sideclique}
 Suppose $G$ is balanced and $p$-critical. Let $z$ be the basepoint of $G$.

 a) If there is at most one exterior component at $z$, then there
 are at least two local components at $z$ which are cliques and have maximum order
 among the local components.

 b) If there is no exterior component at $z$, then there
 are at least three local components at $z$ which are cliques and have maximum order
 among the local components.
  \end{thm}

 \begin{proof}
  Select a minimal representation $\a$ of $G$.
 As in the proof of Theorem \ref{wt}, look at the supports of the
 local components.  These are disjoint intervals, ordered from left
 to right.  Call the leftmost and rightmost components the {\em side
 components.}  The other components are {\em inner components}.
 By hypothesis, at most one local component can be exterior, so at least one of
 the side components is non-exterior.  Call such a component $A$.
 For concreteness, suppose $A$ in on the right side.
 The weight is determined by adding the orders of the non-side
 components.  Since $\a$ is minimal and $G$ is balanced, the
 impropriety equals the sum of the orders of the inner components.
 Hence $A$ cannot contribute to the impropriety.  Now consider $v \in A$.
 Since $A$ is not exterior $I_v \cap I_z \ne \emptyset$.
 Thus $\ell_v \le r_z$.  Since $A$ does not contribute to the impropriety,
 $I_v$ is not contained in $I_z$.  Since $A$ is on the right side,
 this says $r_z < r_v$.  Combining these inequalities, we find $r_z \in I_v$ for all $v \in A$, so $A$ is a clique.

 If there are no exterior components, the above argument shows that
 both the right and left side components must be cliques.

 Now let $A$ and $B$ the side components.  If one of these is
 exterior, by symmetry it may be assumed to be $B$.  Thus from the
 way that weight is defined and because $\a$ is a minimal representation, it follows that $A$ is a component of
 maximum order.  If there are no exterior components, then, by symmetry, $A$ can be
 assumed to have order greater than or equal to $B$.  Thus in
 either case, we can assume that $A$ is local component of maximum order.

 Suppose $x \in A$. Since $G$ is $p$-critical, it follows that removing $x$ will decrease the impropriety.
 That is, we need to find a representation of $G\setminus\{x\}$
 which has a lower impropriety.  Any representation consists of the local components strung out in some order along $I_z$.
 Rearranging the inner components among themselves or changing the
 way they are represented will not decrease the number of intervals
 contained in $I_z$.  Thus some inner component must trade places
 with one of the two side components.  If exchanging an inner component for $B$ has
 a helpful effect, this helpful effect would be present even if $x$ is left in $A$.
 That is, this move could be used to give a representation for $G$ with a
 smaller impropriety, contrary to the minimality of $\a$.
 Thus the essential move is exchanging an inner component $C$ for $A \setminus \{x\}$.

 Suppose $A$ has order $m$ and $C$ has order $n$.
 This exchange increases the number of intervals contained in $I_z$
 by $m-1$ and decreases it by at most $n$.  The inequality here arises
 if $C$ is not a clique, so that some of its intervals must
 intersect $I_z$ while avoiding other intervals from $C$.
 This would force some intervals arising from $C$ to be wholly contained in $I_z$.

 Now $n \le m$ since $A$ has maximum order.  The decrease $d$ in
 impropriety thus satisfies $d = n-(m-1) \le 1$.  Conversely, $d \ge 1$ since $G$ is $p$-critical.
 Thus $n-(m-1) = d = 1$, so $n = m$. And this occurs iff
 all intervals in $C$ can be moved out of $I_z$ --- that is, $C$ is a clique.

 Thus we have shown that there must be one side component $A$ that has
 maximum order and is a clique.  Moreover, there must be an inner component $C$
 that has maximum order and is a clique.  If the type is 0, then $B$ exists
 and, as shown above, $B$ must be a clique.  If it is not of maximum order,
 interchanging $B$ and $C$ would reduce the impropriety of the
 representation, contrary to the assumption that $\a$ is maximal.
 \end{proof}


 \section{Construction of Balanced Interval Graphs}

 Let $z$ denote an isolated vertex.  Let $\cH := H_1, H_2, H_3, \dots, H_n$ denote a sequence of interval graphs.
 Let $\BAL_0(\cH)$ denote the join of $z$ with
 the disjoint union of the $H_i$.  That is, $z$ is made adjacent to all vertices in all of the $H_i$.
 This is clearly an interval graph: represent $z$ by a long interval
 and draw representations of the $H_i$ in disjoint subintervals of this long interval.
 A {\em pendant $P_3$} at $z$ is a path $xyz$ such that $y$ is adjacent only to $z$ and $x$
 and $x$ is adjacent only to $y$.
 If in addition the maximum order of the $\cH_i$ is at least 2,
 $\BAL_k(\cH)$ denotes $\BAL_0(\cH)$ with $k \ge 1$ pendant $P_3$'s attached to $z$.

 \begin{thm}
 \label{balchar}
 A graph $G$ is $p$-critical and balanced iff

 a) $G$ is isomorphic to $\BAL_0(\cH)$ where
 three of the $H_i$ having maximum order are cliques;

 b) $G$ is isomorphic to $\BAL_1(\cH)$ where
 two of the $H_i$ having maximum order are cliques;

 c) $G$ is isomorphic to $\BAL_2(\cH)$ for interval graphs $H_i$.
 \end{thm}

 \begin{proof} If $G$ is $p$-critical and balanced, then by Theorems 3.1 and 3.3,
 $G$ has the form specified above.  For the converse, suppose
 $G$ has the form specified above. It is convenient to assume that $\cH := H_1, H_2, H_3, \dots, H_n$
 is ordered so that $|H_i| \le |H_{i+1}|$ and among the $H_i$ of
 maximum order, the cliques come last.

 If $k = 2$, construct a representation $\a$ of $G = \BAL_2(\cH)$ by putting the two pendant $P_3$'s
 at either ends of a long interval $I_z$ for $z$. Represent the $H_i$ inside smaller subintervals
 of $I_z$.  The weight of $z$ in $G = \BAL_2(\cH)$ is clearly $\S := \sum_{i = 1}^n |H_i|$.
 This is also the impropriety of $z$ in the representation $\a$.
 Thus $\S = \wt(z) \le \wt(G) \le \imp(G) \le \imp(\a) = \S$.
 Therefore, $\wt(G) = \imp(G)$, so $\BAL_2$-graphs are balanced.

 To show $\BAL_2$-graphs are critical, it suffices to show that if
 any interval from the representation $\a$ is removed, then
 the remaining intervals can be rearranged to reduce the impropriety.
 An inner interval contributes directly to the impropriety, so its
 removal reduces the impropriety.  Thus consider a pendant $P_3$ $xyz$.
 If $y$ is removed, then $H_n$ can be moved to where $I_y$ was.  This decreases the impropriety by $|H_n|$.
 If $x$ is removed, then the interval $I_y$ for $y$ can be exchanged for
 $H_n$.  This reduces the impropriety by $|H_n|-1$. But $|H_n|$ is maximal,
 and by definition of $\BAL_2$, there is a local component with at
 least two vertices.  Thus $|H_n|-1 > 0$, so the impropriety does go down.

 If $k = 1$, put the pendant $P_3$ to the left of a long interval $I_z$ for $z$.
 Put small intervals for $H_n$, all containing the right endpoint of $I_z$.
 As before, represent the remaining $H_i$ in smaller intervals contained in $I_z$.
 The weight of $z$ in $G$ is $\sum_{i = 1}^{n-1} |H_i|$.
 This is again $\imp(\a)$.  As in the case $k=2$, this implies $\BAL_1$-graphs are balanced.

 In showing criticality, pendant $P_3$'s and inner intervals can be treated the same way as for $k = 2$.
 If a vertex is removed from $H_n$, then we can exchange $H_n$ for
 $H_{n-1}$ which is an interior clique of the same order as $H_n$ by hypothesis.
 This reduces the impropriety by 1.

 If $k = 0$, $H_n$ and $H_{n-1}$ go on the ends. Removing an
 interior interval obviously reduces the impropriety as before.
 If an interval is removed from one of the end clique components, it can be exchanged for $H_{n-2}$.
 \end{proof}



\begin{thebibliography}{99}

 \bibitem{BJ} J. Beyerl and R. E. Jamison,
 Minimal forbidden subgraphs for $p$-improper
 interval graphs, in preparation.

 \bibitem {CJL} N. Calkin, R. E. Jamison and John B. Light,
 Odd interval graphs, in preparation.

 \bibitem {CH} V. Chv\'{a}tal and P. L. Hammer,
 Aggregation of inequalities in integer programming,
 {\em Annals of Discrete Math.} {\bf 1} (1977), 145-162.

 \bibitem {Coz} M. B. Cozzens and F. S. Roberts,
 Computing the boxicity of a graph by covering its complement by cointerval graphs.
 {\em Discrete Appl. Math.} {\bf 6} (1983), no. 3, 217--228.

 \bibitem {EFKS} Nancy Eaton, Zoltan F\" uredi, Alexander V. Kostochka, and Josef Skokan,
 Tree representations of graphs,
 {\em European J. Combin.} {\bf 28} (2007), no. 4, 1087--1098.

 \bibitem {G} Fanica Gavril,
  The intersection graphs of subtrees of a tree are exactly the chordal graphs,
  {\em J. Combin. Th.} Ser. B {\bf 16} (1974),
  47--56.

 \bibitem {GH} P. C. Gilmore and A. J. Hoffman,
 A characterization of comparability graphs and of interval graphs,
 {\em Canad. J. Math.} {\bf 16} (1964), 539 -- 548.

 \bibitem {GJ-dm} M. C. Golumbic and R. E. Jamison,
 Edge and vertex intersection of paths in a tree,
 {\it Discrete Math.}, {\bf 55}(1985) no. 4, 151 --159.

 \bibitem {GJ-jct} M. C. Golumbic and R. E. Jamison,
 The edge intersection graphs of paths in a tree,
 {\it J. Combin. Th.}, Ser. B {\bf 38}(2006) 8 -- 22.

 \bibitem {GJ} M. C. Golumbic and R. E. Jamison,
 Rank tolerance graph classes,
 {\it J. Graph Theory}, {\bf 52}(2006) no. 4,  317 - 340.

\bibitem {GJT} M. C. Golumbic, R. E. Jamison, and A. N. Trenk,
Archimedean $\phi$-tolerance graphs, {\em J. Graph Th.}, {\bf
41}(2002), 179--194.

\bibitem {GM} M. C. Golumbic and  C. L. Monma, A generalization
of interval graphs with tolerances, {\em Congressus Numer.} {\bf 35}
(1982), 321-331.

\bibitem {GMT} M. C. Golumbic, C. L. Monma, and W. T. Trotter,
Tolerance graphs, {\em Discrete Applied Math.} {\bf 9}
(1984),157-170.

\bibitem {GT}  M. C. Golumbic and A. N. Trenk,
{\em Tolerance Graphs}, Cambridge University Press, Cambridge, 2004.

 \bibitem {H}  G. Hajos,
 \" Uber eine Art von Graphen,
 {\em Internat. Math. Nachr,} {\bf 11} (1957), Problem 65.

 \bibitem {JLL}  M. S. Jacobson, J. Lehel, and L. Lesniak,
 $\phi$-threshold and $\phi$-tolerance chain graphs,
 {\em Discrete Applied Math.} {\bf 44} (1993), 191 -- 203.

\bibitem {JMcMM}  M. S. Jacobson, F. R. McMorris, and H. M. Mulder,
An introduction to tolerance intersection graphs, In Y. Alavi, G.
Chartrand, O. Oellermann, and A. Schwenk, editors, {\em Proc. Sixth
Int. Conf. on Theory and Applications of Graphs}, volume 16, pages
705--724, 1991.

\bibitem {JMS}  M. S. Jacobson, F. R. McMorris, and E. R. Scheinerman,
General results on tolerance intersection graphs, {\em J. Graph Th.}
{\bf 15} (1991) 573 - 577.

 \bibitem {JMdm} R. E. Jamison and H. M. Mulder,
 Tolerance intersection graphs on binary trees with constant
 tolerance 3, {\em Discrete Math} {\bf 215}(2000), 115--131.

 \bibitem  {JMcn} R. E. Jamison and H. M. Mulder, Constant tolerance
 representations of graphs in trees, {\em Congressus Numerantium}
 {\bf 143}(2000), 175--192.

 \bibitem {JM05} R. E. Jamison and H. M. Mulder,
 Constant tolerance intersection graphs of subtrees of a tree,
 {\em Discrete Math.} {\bf 290}(2005) no. 1, 27--46.

  \bibitem {K} Victor Klee,
    Research Problems: What Are the Intersection Graphs of Arcs in a Circle?
    {\it Amer. Math. Monthly} {\bf 76} (1969), no. 7, 810--813.

 \bibitem {LB} C. G. Lekkerkerker and J. C. Boland,
 Representation of finite graphs by a set of intervals on the real line,
 {\em Fund. Math.} {\bf 51} (1962), 45-64.

 \bibitem {mcmc}  Terry A. McKee and F. R. McMorris,
 {\em Topics in Intersection Graph Theory}, SIAM Monographs on
 Discrete Mathematics and Applications, Society for Industrial and
 Applied Mathematics, publ., Philadelphia, 1999.

 \bibitem {MRT} C. Monma, B. Reed, and W. T. Trotter,
 Threshold tolerance graphs, {\em J.  Graph Th.} {\bf 12}
 (1988),  343 - 362.

 \bibitem{PT} Andrzej Proskurowski and Jan Arne Telle,
 Classes of graphs with restricted interval models,
 {\it Discrete Mathematics and Theoretical Computer Science}, {\bf 3}(1999), 167-176.

 \bibitem {R} F. S. Roberts,
 Indifference graphs,
 In Harary, F., editor, {\em Proof Techniques in Graph Theory},
 pp. 139--146. Academic Press, New York.

 \bibitem {S} M. M. Syslo,
 Triangulated edge intersection graphs of paths in a tree,
 {\em Discrete Math.} {\bf 55}, 217--220.

 \bibitem {T} W. T. Trotter,
 A characterization of Roberts' inequality for boxicity.
 {\em Discrete Math.}  {\bf 28} (1979), no. 3, 303--313.

 \end{thebibliography}
\end{document}